\documentclass[reqno,10pt]{amsart}

\usepackage[utf8]{inputenc,xcolor}

\newcommand{\tensor}{\otimes}

\newcommand{\op}{\mathcal}

\newcommand{\cdc}{,\dots,}
\newcommand{\tdt}{\tensor\dots\tensor}
\newcommand{\oneton}{\{1\cdc n\}}

\usepackage{amsmath}%
\usepackage{amsthm}
\usepackage{amsfonts}%
\usepackage{amssymb}%
\usepackage{graphicx}
\usepackage{xy,amsthm,enumerate,xypic,array}  

\input{xy}
\xyoption{all}

\numberwithin{equation}{section}

\newtheorem{theorem}{Theorem}[section]
\theoremstyle{plain}

\newtheorem{corollary}[theorem]{Corollary}
\newtheorem{lemma}[theorem]{Lemma}

\theoremstyle{definition}
\newtheorem{definition}[theorem]{Definition}

\newtheorem{remark}[theorem]{Remark}

\allowdisplaybreaks[2]

\addtocounter{MaxMatrixCols}{2}

\begin{document}

\title{Koszul Feynman Categories} 
\author{Ralph M. Kaufmann and Benjamin C. Ward}
\email{rkaufman@math.purdue.edu, benward@bgsu.edu}


\begin{abstract}  A cubical Feynman category, introduced by the authors in previous work, is a category whose functors to a base category $\op{C}$ behave like operads in $\op{C}$.  In this note we show that every cubical Feynman category is Koszul.  The upshot is an explicit, minimal cofibrant resolution of any cubical Feynman category, which can be used to model $\infty$ versions of generalizations of operads for both graph based and non-graph based examples.\end{abstract}

\maketitle

\section{Introduction}

The bar and cobar constructions for algebras and operads are central tools in homotopical algebra.  They often permit the construction of resolutions which are both cofibrant (hence homotopy invariant) and minimal (hence tractable).

Bar-cobar duality was developed first for categories of algebras over certain operads, where it was central to the study of rational homotopy theory.  This duality was then lifted to the category of the operads themselves \cite{GK}, which proved a fruitful perspective not only for resolving algebraic structures but also to give a conceptual explanation for bar/cobar duality one level down.  The key notion of {\it op.\ cit.\ }is that of a Koszul operad, which ensures that such a resolution can be constructed via the (co)bar construction of its Koszul dual.

More recently there has emerged a need to lift these results still one level higher to capture a tractable, homotopy invariant generalization of the notion of operads themselves.  In addition to operads, a plethora of generalizations of operad-like structures has arisen in recent years and in each case it would be desirable to have access to such a homotopy-invariant analog.

One is thus faced with the following fundamental questions:

\begin{enumerate}
	\item What category of objects encodes operads and generalizations of operads?
	\item Can one construct cofibrant bar/cobar resolutions in such a category?
	\item Which objects in such a category are Koszul?
\end{enumerate}

The first question has been satisfactorily resolved through the work of many authors, who have given independent but not unrelated answers to this question.  These answers include the notions of patterns \cite{Getops}, groupoid colored operads \cite{Pet}, operadic categories \cite{BM0}, as well as our notion of Feynman categories \cite{KW}.  

For both the second and third question, the key notion which makes our answer possible is that of a {\it cubical} Feynman category \cite[Definition 7.2.2]{KW}.  Roughly speaking a cubical Feynman category is a Feynman category where the morphisms can be given a compatible degree.  In the Feynman category encoding operads this degree counts the number of edges in a tree, but in general a cubical Feynman category needn't have any underlying connection to graphs.

In \cite{KW} we defined bar and cobar constructions for representations of a cubical Feynman category and also showed that the category of representations valued in differential graded vector spaces over a field of characteristic $0$ is a model category for which the cobar construction is cofibrant. We may therefore provide an answer to question (2) above by encoding a Feynman category as a representation of a cubical Feynman category.  This is not particularly difficult, and we demonstrate how to do this in Section $\ref{FCSec}$.  The upshot is that the linearization of any Feynman category is an object in a model category admitting bar/cobar duality, which provides an affirmative answer to question (2).

The main purpose of this article is to answer the third question.  We prove:

\begin{theorem}\label{thmintro}  Every cubical Feynman category is Koszul.
\end{theorem}
This reduces the question of Koszulity to a verifiable condition about degrees of morphisms.  This result allows us to construct a minimal cofibrant resolution of any cubical Feynman category, including for familiar graph based generalizations of operads, but its applicability is not confined to graph based examples.  These resolutions in turn model $\infty$ versions of the original structure, thereby encoding notions of
$\infty$-operads, $\infty$-cyclic operads, $\infty$-(wheeled) properads, $\infty$-modular operads, etc.\ in particular cases of interest.

Theorem $\ref{thmintro}$ has been proven in several particular graph based examples, e.g.\ \cite{WardMP}, \cite{MB2} and \cite{MB3}.  In these examples, the Koszul property reduces to a contractibility of a fiber which may be identified with the graph associahedron (after \cite{CDev}) of a dual graph, see \cite[Section 3.4]{WardMP}.  However, our proof is a novel and comparably elementary one, even for these particular cases.  To highlight the idea of this proof, in Section $\ref{permsec}$ we give a self-contained and elementary spectral sequence argument showing the permutohedra are acyclic.  Recall a permutohedron is a polytope whose vertices may be viewed as directed edge paths connecting opposite vertices of a hyper-cube (whence the terminology ``cubical'').   In a cubical Feynman category such edge paths need not correspond to distinct operations, but the above mentioned argument readily generalizes to account for this.   Thus we realize the fiber of the Koszul map not as the blow-up of a simplex, but rather as the blow-down of a permutohedron, a perspective which simplifies the needed Koszulity argument considerably.

\subsection*{Conventions} We work over a field $k$ with $\text{char}(k)= 0$.  Chain complexes are graded homologically.  Given a set $A$ we write Det$(A)$ for the top exterior power of the span of $A$.  We write $S_n$ for the symmetric group of automorphisms of $\oneton$.  By a mod 2 order on a set we refer to an orbit of a totally ordered set by the action of the subgroup of even permutations.  For $r\in \mathbb{Z}$, we write $\Sigma^r$ for the operator which shifts a graded vector by $r$.  Explicitly $(\Sigma^r V)_m = V_{m-r}$ and write $\Sigma:=\Sigma^1$.


\section{From Permutohedra to Simplices.}\label{permsec}
We begin with a warm-up version of our main theorem which requires no operadic prerequisites.  It concerns a well known family of polytopes, the permutohedra.  Let us give an elementary argument that the permutohedron is acyclic.  We emphasize the purpose of this section is not to derive the result -- which is well known --  but rather to highlight an algebro-combinatorial proof of this result which can be generalized to prove our main theorem.

By a totally ordered (TO) partition we refer to a partition along with a total order on its set of blocks.
\begin{definition}\label{pdef}  Define $C_\ast(P_n)$ to be the following chain complex.  A basis of $C_r(P_n)$ is given by the TO partitions of $\oneton$ having $n-r$ blocks.  In particular $C_r(P_n)$ is non-zero for $0 \leq r \leq n-1$.
	The differential $C_r(P_n) \stackrel{d}\to C_{r-1}(P_n)$ is given by a signed sum of ways to subdivide blocks in a given TO partition, The signs in the differential are determined by 
	  the Koszul sign rule, where the symbol $>$ has degree $1$, and is added from the right.  For example:
\begin{eqnarray}
d(\{1,3\}>\{2,4\}) &  = & -\left(\{1\}>\{3\}>\{2,4\}\right) - (\{3\}>\{1\}>\{2,4\}) \nonumber \\ \nonumber &&  + (\{1,3\}>\{2\}>\{4\}) + (\{1,3\}>\{4\}>\{2\}).	
\end{eqnarray}	  
  \end{definition}
Using this convention, one easily verifies that:
\begin{lemma} As defined above, $(C_\ast(P_n),d)$ is indeed a chain complex, i.e.\ $d^2=0$.  
\end{lemma}

The chain complex $(C_\ast(P_n),d)$ is isomorphic to the cellular chains on the permutohedron on $n$ letters $P_n$, so we call a basis element a cell.  A commonly used notation for such a cell is $13|24$, see e.g.\ \cite{KZhang} and references therein.  The current notation emphasizes the order on the blocks, which determines the signs in the differential in Definition $\ref{pdef}$. 
  Note that while these signs are natural in our context, they do not arise from a choice of orientation of the cells of the permutohedron $P_n$.  

Define the {\it width of a cell} to be the cardinality of its least block.  
The differential of $C_\ast(P_n)$ can not increase the width, so there is a filtration $C_\ast(P_n)^{1} \subset C_\ast(P_n)^{2} \subset ...\subset C_\ast(P_n)^{n} = C_\ast(P_n)$, where $C_\ast(P_n)^{t}$ denotes the subcomplex spanned by cells of width $\leq t$.    The width filtration was studied topologically under the name initial branching number in \cite{KZhang}.

Consider the spectral sequence associated to this filtration.  The $E^0$ page is bigraded by $p$, the width, and $q$, the cellular degree minus the width.  It has differential $d^0$ which sees only those terms that preserve width (as opposed to lower it).  For example, when $n=4$ the $E^0$ page is comprised of four rows of total dimension $1+4+6(1+2)+4(1+6+6)$ arranged by the shape of the partitions as:
\smallskip
\begin{center}
	\begin{tabular}{l|c|c|c|c||l}
		$_p$ \textbackslash $^{p+q}$ & 3 & 2 &  1 & 0  & row is \\ \hline
	    4 & (4) & - & - & -&  $\Sigma^{3}$ a point \\
		3 & - &	(1,3) & - & - & $\Sigma^{2}$ 4 points \\
		2 & -  &	(2,2) & (1,1,2) & - & $\Sigma$ 6 intervals  \\
		1 & - &	(3,1) & (2,1,1) and (1,2,1) & (1,1,1,1) & 4 hexagons\\
	\end{tabular}
\end{center}
\smallskip
Here the shape of a TO partition is a tuple of integers in parentheses indicating the sizes of the ordered list of blocks, with least block on the right by convention.  See the left hand side of Figure $\ref{fiber2}$ for a picture of this filtration.



The number of summands in the splitting of a row is given by a choice of least block, hence the binomial coefficients.  Inductively, each row has homology only on the diagonal, hence the $E^1$ page is a single chain complex whose dimensions match those of an $n-1$ simplex.  Moreover the differential on the $E^1$ page is given by a signed sum of ways to split off an element from the least block, and so matches the differential in the simplicial chain complex of the simplex.

More precisely, this spectral sequences has the following form.   Since this Lemma is a special case of Lemma $\ref{widthfiltration}$ proven below, we defer the proof.


\begin{lemma}\label{Plem}  The spectral sequence $E$ associated to the filtered chain complex $C_\ast(P_n)$ has the following form:
	\begin{enumerate}
		\item There is an isomorphism of chain complexes:
		$ E^0_{p,\ast} \cong \displaystyle\bigoplus_{\substack{Y\subset \oneton \\ n-|Y| = p}} \Sigma^{-1} C_\ast(P_{Y}).$
		
		\item In particular $E^1_{p,q}$ has rank ${n \choose p}$ if $q=-1$ and is $0$ otherwise.

		\item The complex $(E^1_{p,-1}, \partial_1)$ is isomorphic to the standard simplicial chain complex of an $n-1$ simplex.

		\item In particular $E^2_{p,q}$ has rank 1 if $p=1$ and $q=-1$ and is $0$ otherwise.
	\end{enumerate}
\end{lemma}

This Lemma makes possible a comparison between $C_\ast(P_n)$ and the simplicial chains on an $n-1$ simplex, albeit with less familiar sign conventions, as we now explain.  Let $C_\ast(\Delta^{n})$ denote the standard simplicial chain complex of an $n$-simplex.  It is spanned by non-empty subsets $X\subset \oneton$, with differential given as the alternating sum of removing an element.  Next, let $C^-_\ast(\Delta^{n})$ denote the ``odd simplicial chain complex'' of an $n$-simplex by which we mean the following.  It is spanned by pairs $(Y, \succ_Y)$ where $Y$ is a subset of $\oneton$ with non-empty complement and $\succ_Y$ is a total order on $Y$.  Two such elements are related by the sign of the permutation which exchanges their order.  The degree of such an element is $n-1-|Y|$.   Finally the differential is given by the sum of all ways to add an element in $Y^c$ as the least (right most) element to the ordered set $(Y, \succ_Y)$.  This differential is square zero because adding two elements of the complement in both orders result in terms which sum to zero.

It is an enjoyable exercise to construct an isomorphism 
\begin{equation}\label{deltaiso}
C_\ast(\Delta^{n}) \stackrel{\cong}\to C^-_\ast(\Delta^{n})
\end{equation}
by sending a set $X$ to a signed multiple of its complement (with the standard linear order).  This sign is the sign of the unique permutation which moves the elements of $X$ to the right of the elements of $Y=X^c$, while preserving the order of both subsets.

There is a surjective map of chain complexes $C_\ast(P_n)\twoheadrightarrow C^-_\ast(\Delta^{n-1})$ given by sending a TO partition to zero unless all of its blocks except possibly the least block are of size $1$ -- in which case it is sent to the complement of its least block, with the order induced from the TO partition.  This is a chain map because splitting off a block of size $1$ from the least block places this new block in the rightmost (least) position among the size $1$ blocks.


Composing this morphism with isomorphism in Equation $\ref{deltaiso}$ and appealing to Lemma $\ref{Plem}$ we have:
\begin{corollary}\label{PtoS}  The composition map $C_\ast(P_n)\twoheadrightarrow C_\ast(\Delta^{n-1})$ is a quasi-isomorphism.  In particular $C_\ast(P_n)$ has the homology of a point.
\end{corollary}


\section{Cubical and Koszul Feynman Categories.}\label{FCSec}  In this section we recall the preliminaries needed to state our main theorem.  We refer to \cite{KW} for full details.  A {\it Feynman category} $(\mathsf{V},\mathsf{F})$ is a pair of a symmetric monoidal category $(\mathsf{F},\tensor)$ and a (non-monoidal) fully faithful subcategory $\mathsf{V}\subset \mathsf{F}$, subject to several axioms that we now describe informally.  First $\mathsf{V}$ is a groupoid.  Second, every object in $\mathsf{F}$ is isomorphic to $v_1\tensor\dots\tensor v_n$ for $v_i \in \text{Ob}(\mathsf{V})$.  Third, every morphism in $\mathsf{F}$ is isomorphic to some $ \tensor_{j=1}^m \phi_j$ where each $\phi_j$ has target in $\text{Ob}(\mathsf{V})$.  See \cite[Definition 1.1.1]{KW} for the full definition.  We often refer to such a pair $(\mathsf{V},\mathsf{F})$ simply as $\mathsf{F}$.

In a Feynman category, morphisms of the form $v_1\tensor \dots \tensor v_n \to v_0$ play an essential role.  In \cite{KW} we called these generating morphisms, here we call them {\it basic morphisms} to avoid possible conflict in terminology.  We use the following notation for the set of basic morphisms.  If $\vec{v}:= (v_1\cdc v_n; v_0)$ is a list of $n+1$ objects of $\mathsf{V}$ we define
$
\mathsf{F}(\vec{v}) : = Hom_\mathsf{F}(v_1\tensor\dots\tensor v_n, v_0).
$


To a Feynman category $(\mathsf{V},\mathsf{F})$ we associate the freely enriched categories $\mathbb{V}\subset\mathbb{F}$ with the same objects by taking the span of each set of morphisms.  
We similarly denote the span of $\mathsf{F}(\vec{v})$ by $\mathbb{F}(\vec{v})$.

\subsection{Cubical Feynman Categories}\label{CFC}
One application of Feynman categories is to encode generalizations of operads.  There is a Feynman category $\mathsf{O}$ whose category of representations $\mathsf{O}\text{-ops}_{\op{C}}:=Fun_\tensor(\mathsf{O},\op{C})$ are operads in $\op{C}$.  There are Feynman categories which similarly encode cyclic operads, modular operads, dioperads, wheeled prop(erad)s among others.  In these examples the morphisms of the respective $\mathsf{F}$ may be regarded as graphs.  
 In such examples, the morphisms of $\mathsf{F}$ have extra structure; a gradation corresponding to the number of edges of the graph.  This motivated the definition of a cubical Feynman category \cite[Definition 7.2.2]{KW}.

A cubical Feynman category is a Feynman category $(\mathsf{V}, \mathsf{F})$ along with two additional structures.  The first is a function $\text{deg}:\text{Mor}(\mathsf{F})\to \mathbb{Z}_{\geq 0}$, which satisfies the following properties:
	\begin{enumerate}
		\item $deg(f\tensor g) = deg(f)+deg(g)$,
		\item $deg(f\circ g) = deg(f)+deg(g)$,
		\item $deg(f)=0$ if and only if $f$ is an isomorphism, 
		\item every basic morphism is a composition of morphisms of degree $0$ and $1$.
	\end{enumerate}
We call such a function a proper degree function.

To define the second structure, suppose we are given a Feynman category with a proper degree function.  Write $C_n(X,Y)$ for the set of chains of $n$ or more composible morphisms in $\mathsf{F}$ which start at $X$ and end at $Y$, such that $n$ morphisms in the chain have non-zero degree, modulo the equivalence relation induced by composing degree $0$ morphisms.  Define $C_n^+(X,Y)\subset C_n(X,Y)$ to be the subset of chains whose constituent morphisms have degree $0$ or $1$.

\begin{definition}\label{cfcdef}  A {\it cubical Feynman category} is a Feynman category along with a proper degree function and a free $S_n$ action on the set $C_n(X,Y)$, compatible with composition of sequences, and such that composition of morphisms defines an isomorphism of sets $C_n^+(X,Y)_{S_n}\stackrel{\cong}\to Hom_\mathsf{F}(X,Y)_n$, where this last subscript denotes the degree of the morphism.
\end{definition}

We emphasize that by {\it composition of sequences} in this definition we mean simply the concatenation of a sequence ending at a given object with a sequence starting at that object.  The compatibility assumption specifically means that the composition of sequences $C_n(X,Y)\times C_m(Y,Z) \to C_{n+m}(X,Z)$ is equivariant with respect to the action of the group $S_n\times S_m$, where the target inherits such an action from the natural injection $S_n\times S_m \hookrightarrow S_{n+m}$.
	
\begin{remark}\label{KDrmk}  As we shall see below, knowing that a Feynman category is cubical endows the category $\mathbb{F}\text{-ops}_{dgVect}$ with the fundamental structures of Koszul duality.  This includes the bar/cobar construction \cite[Definition 7.4.1]{KW}
 as well as the notion of quadratic objects and quadratic duality \cite[Appendix A.3]{W6}.	These constructions are simplified by the assumption that each $Hom_\mathsf{F}(X,Y)$ is a finite set.  We impose this restriction from now on and refer to \cite[Definition 7.3.1]{KW} and \cite[Section 2.4]{WardMP} for ways to weaken this hypothesis.
\end{remark}

\subsection{Quadratic duality.}
A $\mathsf{V}$-colored tree will refer to a rooted tree whose flags are colored by objects of $\mathsf{V}$, such that two flags which form an edge have the same color.  We furthermore assume that the leaves of such a tree are numbered $1$ to $n$.  If $\mathsf{t}$ is a $\mathsf{V}$-colored tree we write $V(\mathsf{t})$ and $E(\mathsf{t})$ for its set of vertices and edges.  Given objects $v_i$ of $\mathsf{V}$, we say a $\mathsf{V}$-colored tree with $n$ leaves is of type $(v_1\cdc v_n; v_0)$ provided the $i^{th}$ leaf is colored by $v_i$ and the root is colored by $v_0$. 


Given a groupoid $\mathsf{V}$ we define a new Feynman category $(\mathsf{Cor_V},\mathsf{T_V})$ as follows.  The objects of $\mathsf{Cor_V}$ are  $\mathsf{V}$-colored corollas with numbered leaves (equivalently lists $(v_1\cdc v_n;v_0)$ of objects in $\mathsf{V}$).  The morphisms of $\mathsf{Cor_V}$ are products of morphisms in $\mathsf{V}$, along with symmetric group actions permuting the labels of the leaves.  The objects of $\mathsf{T_V}$ are lists of objects in $\mathsf{Cor_V}$ (as is the case for any Feynman category).  Finally the basic morphisms of $\mathsf{T_V}$ are $\mathsf{V}$-colored trees, whose edges (both internal and external) are labeled by automorphisms of the color of the edge label.  Composition of morphisms is given by insertion of a tree at a compatibly colored vertex, composing automorphisms of the edge labels.

\begin{remark}
\label{rmk:plus}
The  groupoid of $\mathsf{V}$-colored corollas and Feynman categories built on it appear in the plus construction, see \cite[Appendix B]{feynmanrep} for full details. There, a vertex decoration is added to yield a plus construction $(\mathsf{V}^+,\mathsf{F}^+)$ for a Feynman category $(\mathsf{V},\mathsf{F})$. Forgetting these vertex decorations gives a forgetful morphism of Feynman categories $(\mathsf{V}^+,\mathsf{F}^+)\to (\mathsf{Cor_V},\mathsf{T_V})$.
\end{remark}

The Feynman category $\mathsf{T_V}$ has a proper degree function given by associating to a tree (resp.\ forest) its number of (internal) edges.  With this degree function, a chain of degree $1$ morphisms specifies a forest with $n$ edges along with an order in which they are assembled.  Permuting this order specifies an action of $S_n$ on such chains of length $n$. 

\begin{lemma}\label{TVcubical}  With these structures, the Feynman category $\mathsf{T_V}$ is cubical.
\end{lemma}

\begin{proof} 

Let us first verify that the stated degree function satisfies the conditions (1)-(4) enumerated above.  Since tree insertion and juxtaposition are both additive operations with respect to the number of edges, the first two conditions are satisfied.  To verify the condition (3), observe that if a morphism is of degree $0$ then it has no edges and hence has an inverse given by the same corolla with inverse automorphism labels.  Conversely if a morphism has non-zero degree it can't be an isomorphism due to condition (2).  Whence condition (3).  Finally, to verify condition (4) note that the condition is vacuous for a basic morphism of degree $0$ or $1$.  We may then proceed by induction, as collapsing an edge of a degree $n$ morphism allows us to write such as a composition of a degree $n-1$ morphism and a degree $1$ morphism.  Here, the automorphism labeling the edge plays no role in the collapse, it merely decorates the edge of the degree $1$ morphism in the decomposition.

Let us now proceed to consideration of the $S_n$ action in more detail. 
Recall that in an operad, given a $\circ_i$ operation and a permutation $\sigma$ such that the composition $\circ_i \cdot \sigma$ is well defined, there is a unique permutation $\sigma^\prime$ and a unique integer $i^\prime$ such that $\sigma^\prime\cdot\circ_{i^\prime} =\circ_i\cdot \sigma$.  In the parlance of \cite{KW}, trees are of ``crossed type'' \cite[Lemma 5.2.1]{KW}.

We may translate this fundamental recollection into our current language as follows.  Write $\mathsf{T}$ in place of $\mathsf{T_V}$ in the case that $\mathsf{V}$ is a groupoid with one object and one morphism.  Then crossed type implies that every degree $1$ morphism in $\mathsf{T}$ can be written uniquely as $\tau\cdot \circ_i$, and thus every chain of $n$ degree $1$ morphisms in $\mathsf{T}$ is equivalent to a unique chain of the form:
$\tau \cdot \circ_{{i_n}}\cdot . . . \cdot \circ_{{i_1}}.$

Such a chain specifies an edge ordered forest and each edge ordered forest is specified by a unique chain of this form.  As such, two chains corresponding to the same edge-ordered forest must be equivalent, and so the $S_n$ action is well defined. To complete the proof in this case, observe that two edge orderings on a forest are related by a permutation of the edges (by definition), and since the $S_n$ action simply permutes the degree $1$ morphisms corresponding to the edges, the action will be transitive.

The case of general $\mathsf{V}$ follows by exactly the same logic, as soon as one observes that the labeled trees which constitute morphisms in $\mathsf{T_V}$ are also of crossed type.  In other words, every degree $1$ morphism can be written in the form $\tau \cdot \circ_{i,\alpha}$ for an isomorphism $\tau$ and a degree $1$ morphism $\circ_{i,\alpha}$, which grafts at the leaf labeled by $i$ of a given vertex while labeling the newly formed edge of color $v$ with the automorphism $\alpha \in Aut(v)$.  In particular, a composition of the form $\circ_{j,\beta}\cdot \sigma $ can be written uniquely in the form $\tau \cdot \circ_{i,\alpha}$ by composing the automorphisms of $\mathsf{V}$ corresponding to the grafted edge to determine $\alpha$ and using the remaining automorphisms of $\mathsf{V}$ to form $\tau$. \end{proof}

Any Feynman category $\mathsf{F}$ with vertices $\mathsf{V}$ may be viewed as a symmetric monoidal functor (called a  $\mathsf{T_V}$-op in \cite{KW}) from the category $\mathsf{T_V}$ to sets via $\vec{v} \mapsto \mathsf{F}(\vec{v})$.  Similarly, its linearization $\mathbb{F}$ may be viewed as a $\mathsf{T_V}$-op valued in Vect.  Informally we may say that $\mathsf{T_V}$ is the Feynman category which encodes the Feynman categories with fixed vertex groupoid $\mathsf{V}$.


The forgetful functor from $\mathsf{T_V}$-ops to $\mathsf{Cor_V}$-modules has a left adjoint \cite[Theorem 1.5.3]{KW}, which we denote by $Fr(-)$.  For each object $\vec{v}$, $Fr(A)(\vec{v})$ is a direct sum over $\mathsf{V}$-colored trees of type $\vec{v}$ whose vertices are labeled by $A$, modulo the action of the automorphism group at each internal edge.  In particular, $Fr(A)(\vec{v})$ has an additional grading given by the number of vertices of the tree which we call the {\it weight}.  Adapting  \cite[Appendix A.3]{W6} to this particular example, we say that a Feynman category is {\it quadratic} if there exists a $\mathsf{Cor_V}$-module $A$ and an object-wise surjective morphism $Fr(A)(\vec{v}) \stackrel{\pi_{\vec{v}}}\to \mathbb{F}(\vec{v})$ whose kernel is generated in weight 2.

\begin{lemma}  If $\mathsf{F}$ is a cubical Feynman category, then its linearization $\mathbb{F}$ is quadratic.  
\end{lemma}

\begin{proof}

	Let $A$ be the collection of morphisms of $\mathbb{F}$ of degree $\leq 1$, viewed as a $\mathsf{Cor_V}$-module in dgVect.  
	There is a morphism of dg $\mathsf{T_V}$-ops $Fr(A) \to \mathbb{F}$ given by composing morphisms along the tree as a flow chart.  Since every morphism in $\mathbb{F}$ can be written as the span of a composition of degree $1$ morphisms, this map is surjective for each $\vec{v} \in \text{Ob}(\mathsf{Cor_V})$.
	

Fix such a $\vec{v}$ and let $\alpha\in \text{ker}(\pi_{\vec{v}})$.  The graded vector space $\mathbb{F}(\vec{v})$ splits over degree, and in each degree $n$ splits over the underlying set of morphisms $\mathsf{F}(\vec{v})$ of degree $n$.  For such a morphism $\phi \in \mathsf{F}(\vec{v})_n$, the cubical condition tells us
	$\pi^{-1}(\phi) = \text{span}\{\sigma\gamma \ | \ \sigma \in S_n \text{ and } \gamma \in C_n^+(v_1\tensor\dots\tensor v_n; v_0) \}.$
	 Therefore $\alpha$ is a sum of vectors of the form $\alpha_\phi:=\sum_{\sigma \in S_n} c_\sigma\sigma\gamma$ where $\sum c_\sigma = 0$, and it suffices to show that each such $\alpha_\phi$ is generated in weight 2.
	 
	 Fixing $\phi$, assume that $\alpha_\phi\neq 0$. 
	 Then $\sum c_\sigma = 0$ implies
	 $$\alpha_\phi
	 =\sum_{\sigma \in S_n\setminus \{id\}} c_\sigma(\sigma\gamma-\gamma),$$
	  so it is in turn sufficient to show that each $\sigma\gamma - \gamma$ is generated in weight 2.  Fix such a $\sigma$ and write it as a product of transpositions $\sigma = \tau_r\cdot ... \cdot \tau_1$.   Then $\gamma - \sigma \gamma = 
	\sum_{j=0}^r (1-\tau_j)(\tau_{j-1} \dots \tau_1 \gamma)$, interpreting $\tau_0$ as $0$, and hence it is sufficient to show that that $\tau\gamma - \gamma$ is generated in weight $2$ for a transposition $\tau$.  However this is immediate, as a transposition switches the order of two edges in the tree underlying the chain of morphisms $\gamma$.  Indeed we may assume the transposition is of the form $\tau = (s \ s+1)$ and use the fact (Definition $\ref{cfcdef}$) that composition of chains of morphisms is compatible with the action of the symmetric group to conclude that $\tau(\gamma_{> s+1}\circ\gamma_{s+1}\circ\gamma_s\circ \gamma_{ < s })$
is of the form 
 $\gamma_{> s+1}\circ\tilde\gamma_{s}\circ\tilde\gamma_{s+1}\circ \gamma_{ < s}$.	Hence $\mathbb{F}$ is quadratic, with quadratic presentation $Fr(A)/\langle \tau\gamma-\gamma \rangle$, where the relations are generated over all chains of morphisms $\gamma$ and transpositions $\tau$, or equivalently over just degree $2$ chains $\gamma$ and the transposition $\tau =(12)$.
\end{proof}

Given such an $\mathbb{F}$, its quadratic dual $\mathbb{F}^!$ is by definition the quadratic $\mathsf{T}_{\mathsf{V}}$-op whose generators $\Sigma^{-1}A^\ast$ are the object-wise linear duals of $A$, with a shift in degree, and whose relations are generated by those functionals vanishing on each $\tau\gamma-\gamma$.  In this case, since each $A(\vec{v})$ has a given basis (namely the degree $1$ elements of the finite set $\mathsf{F}(\vec{v})$), it may be canonically identified with its linear dual.  Under these identifications we have $\mathbb{F}^!\cong Fr(\Sigma^{-1}A)/\langle \Sigma^{-2}(\tau\gamma+\gamma) \rangle $.

\begin{remark}  The presentation of the quadratic dual $\mathbb{F}^!$ above shows that it coincides with the $\mathfrak{K}$-twist of $\mathbb{F}$, as defined in \cite[Definition 5.2.4]{KW}.
\end{remark}


\subsection{(Co)bar construction of a cubical Feynman category and Koszulity.}

Although it is possible to extract the definition of the (co)bar construction of a cubical Feynman category by combining \cite[Definition 7.4.1]{KW} with Lemma $\ref{TVcubical}$ above, we will take some space here to give a more careful unpacking of this definition.  Specifically, if $(\mathsf{V},\mathsf{F})$ is a cubical Feynman category and $\mathbb{F}$ its linearization, let us unwrap the definition of the cobar construction of the linear dual of $\mathbb{F}$, denoted $\Omega(\mathbb{F}^\ast)$.

Let $\mathsf{t}$ be a $\mathsf{V}$-colored tree. If $w$ is a vertex of $\mathsf{t}$ with output color $v_0$, we may enumerate the colors labeling its inputs $v_1\cdc v_n$ and define
\begin{equation*}  \mathbb{F}(w):=[\bigoplus_{\sigma \in S_n} Hom_\mathbb{F}(v_{\sigma(1)}\tensor \dots \tensor v_{\sigma(n)}; v_0)]_{S_n}.
\end{equation*} 

If $e$ is an edge of $\mathsf{t}$ of color $v_e\in \text{ob}(\mathsf{V})$, then the group $Hom_\mathsf{V}(v_e,v_e)$ acts on $\tensor_{w\in V(\mathsf{t})}\mathbb{F}(\mathsf{t})$ by simultaneously composing with the automorphism on the input factor and its inverse on the output factor. Let $E_\mathsf{t}$ be the group $\times_{e\in E(\mathsf{t})} Hom_\mathsf{V}(v_e,v_e)$.  Define
\begin{equation}\label{coinvariants}
\mathbb{F}(\mathsf{t}) := [\bigotimes_{w \in V(\mathsf{t})} \mathbb{F}(w)]_{E_{\mathsf{t}}}
\end{equation}
Observe that the group $E_{\mathsf{t}}$ is finite, and the vector space $\mathbb{F}(w)$ is finite dimensional (Remark $\ref{KDrmk}$).  We then define $(\Omega(\mathbb{F}^\ast),\partial) \in \mathsf{T_V}\text{-ops}_{\text{dgVect}}$ by the formula
\begin{equation}\label{cobar}
\Omega(\mathbb{F}^\ast)(\vec{v}) = \bigoplus_{\mathsf{t} \text{ of type }\vec{v}} \Sigma^{-1}Det^{-1}(E(\mathsf{t}))\tensor \mathbb{F}(\mathsf{t})^\ast \cong \bigoplus_{\mathsf{t} \text{ of type }\vec{v}} Det^{-1}(V(\mathsf{t}))\tensor \mathbb{F}(\mathsf{t})^\ast,
\end{equation}
along with the $\mathsf{T_V}$-op structure given by grafting trees and a differential $\partial$ which expands tree edges in all possible ways.  Here $Det^{-1}$ takes the top exterior power of a set, but placed in negative degree, and $\Sigma^{-1}$ indicates a downward shift in degree.  
The isomorphism in equation $\ref{cobar}$ is induced by the isomorphism $\Sigma^{-1}Det^{-1}(E(\mathsf{t}))\cong Det^{-1}(V(\mathsf{t}))$ given by identifying an internal edge of a tree with the vertex immediately above it and placing the root vertex, by convention, in the last position.


Finally we observe that there is a natural map of $\mathsf{T_V}$-ops
\begin{equation}\label{koszulmap}
\Omega(\mathbb{F}^\ast) \to \mathbb{F}^!.
\end{equation}
 This map has the general description that a tree labeled by degree 1 morphisms is sent to its composite, else it's sent to zero.  The shift $\Sigma^{-1}$ at each vertex means the input comes with a mod 2 order on the set of vertices, hence basic morphisms.  By construction, the sum of expansions will land in a relation of the quadratic dual hence ensuring the result is a dg map.

The following definition generalizes the notion of a Koszul operad \cite{GK}.
\begin{definition}\label{Koszuldef}  A quadratic Feynman category is Koszul if the map in Equation $\ref{koszulmap}$ induces a homology isomorphism $H_\ast(\Omega(\mathbb{F}^\ast))(\vec{v})\cong \mathbb{F}^!(\vec{v})$ for each $\vec{v}$ in Ob($\mathsf{Cor_V}$).
\end{definition}



\section{Cubical implies Koszul.}

This section is devoted to the proof of our main theorem:

\begin{theorem}\label{main} Let $(\mathsf{V},\mathsf{F})$ be a cubical Feynman category and let $\mathbb{F}$ be its linearization.  Then $\mathbb{F}$ is Koszul.
\end{theorem}
After Definition $\ref{Koszuldef}$, it suffices to prove that for each list $\vec{v}=(v_1\cdc v_n; v_0),$ the map $\Omega(\mathbb{F}^\ast)(\vec{v}) \stackrel{\sim}\to \mathbb{F}^!(\vec{v})$ is a quasi-isomorphism.  For any $\mathsf{V}$-colored tree, the vector space $\mathbb{F}(\mathsf{t})$ has a basis given by those trees whose vertices are labeled by morphisms in the set valued Feynman category $\mathsf{F}$.  Call the set of such vectors $\mathsf{F}(\mathsf{t})\subset \mathbb{F}(\mathsf{t})$ and call an element in this set homogeneous.  We emphasize that a homogeneous vector is an equivalence class of labeled trees under the action of the group $E_\mathsf{t}$.
 
Using this basis we identify  
\begin{eqnarray*}
	\Omega(\mathbb{F}^\ast)(\vec{v}) \cong  \bigoplus_{\mathsf{t} \text{ of type }\vec{v} } Det^{-1}(V(\mathsf{t}))\tensor\mathbb{F}(\mathsf{t}) 
\end{eqnarray*}%
An element in $Det^{-1}(V(\mathsf{t}))\tensor\mathbb{F}(\mathsf{t})$ can be written uniquely as a sum of pure tensors whose right hand factor is homogeneous.  For each such pure tensor, composition of morphisms in $\mathsf{F}$ gives us a map of sets $\mathsf{F}(\mathsf{t})\to \mathsf{F}(\vec{v})$.  This in turn gives us a splitting indexed over the set $\mathsf{F}(\vec{v})$:
\begin{equation*}
\Omega(\mathbb{F}^\ast)(\vec{v}) \cong \bigoplus_{\phi \in \mathsf{F}(\vec{v})} C_\ast(\phi),
\end{equation*}
by defining $C_\ast(\phi)$ to be the span of such vectors which compose to $\phi\in \mathsf{F}(\vec{v})$.   Since the cobar differential is a sum over ways to decompose a morphism, which doesn't alter the composition, this is indeed a splitting. 

\subsection{Description of the complex $C_\ast(\phi)$}  Fix a (basic) morphism $\phi \in \mathsf{F}(\vec{v})_n$.  The complex $(C_\ast(\phi),d)$ is the span of homogeneous vectors represented by trees of type $\vec{v}$  whose vertices are labeled by basic morphisms of $\mathsf{F}$ of non-zero degree, along with a choice of mod 2 order on the set of tree edges (or equivalently the tree vertices), and with the property that reading the tree as a flow chart (disregarding this edge order) gives the morphism $\phi$.  Note that the same tree with the opposite mod 2 order on the set of edges is identified with the negative of the original.  The differential is the sum over edge expansions, with convention that an expanded edge is placed in the last position with respect to the mod 2 edge order.

The degree of a (homogeneous) vector in $C_\ast(\phi)$ is given by the negative of the number of tree edges minus 1, or equivalently the negative of the number of tree vertices.  The map $C_\ast(\phi)\subset \Omega(\mathbb{F}^\ast)\to \mathbb{F}^!$ is $0$ away from degree $-n$.  In degree $-n$, the $n$ vertices are necessarily each labeled with a basic morphism of degree $1$, and the mod 2 order on the vertices specifies a mod 2 order on the these degree 1 generating morphisms, and hence specifies an element of $\mathbb{F}^!$, to which it is sent under the map in Equation $\ref{koszulmap}$.

\subsection{Analysis of the width filtration}

Define the {\it width} of a homogeneous element in $C_\ast(\phi)$ to be the degree of the morphism which labels the root vertex.
The differential of $C_\ast(\phi)$ can not increase the width, so there is a filtration of the form $C_\ast(\phi)^1 \subset C_\ast(\phi)^2 \subset ...\subset C_\ast(\phi)^n = C_\ast(\phi)$, where $C_\ast(\phi)^t$ denotes the subcomplex spanned by cells of width $\leq t$.  The following Lemma generalizes Lemma $\ref{Plem}$ above.


\begin{lemma}\label{widthfiltration} We continue to assume $\phi$ is a basic morphism of $\mathsf{F}$ of degree $n$.  The spectral sequence $E=E(\phi)$ associated to the filtered chain complex $C_\ast(\phi)$ has the following description:
	\begin{enumerate}
		\item There is a splitting of chain complexes:
		$E^0_{p,\ast} \cong \displaystyle\bigoplus_{\substack{X\subset \{1\cdc n\} \\ |X| = p\geq 1}} C_\ast(\phi,X), $
		\medskip	
		where $C_\ast(\phi,X)$ has homology
$		H_q(\phi,X) \cong \begin{cases}
		k & \text{ if } q= -n-1  \\
		0 & \text{ else }.
		\end{cases}$
		\medskip

	\item In particular $E^1_{p,q}$ has rank ${n \choose p}$ if $q=-n-1$ and is $0$ otherwise.

		\item The complex $(E^1_{p,-n-1}, \partial_1)$ is isomorphic to a simplicial chain complex of an $n-1$ simplex, shifted down by degree $1$.
	
		\item In particular $E^2_{p,q}$ has rank $1$ if $p=1$ and $q=-n-1$ and is $0$ otherwise.

	\end{enumerate}
\end{lemma}
\begin{proof}  
We proceed by induction on $n$.  If $\phi$ has degree 1, the statements are vacuous, so we assume the statements for all morphisms of degree less than $n$. Fix a decomposition of $\phi$ into a chain of degree $1$ morphisms  $\phi= \phi_n\circ... \circ \phi_1$.  After Subsection $\ref{CFC}$, such a chain represents an equivalence class in $C_n^+(v_1\tdt v_r; v_0)$, which we denote by $[\phi_n\circ... \circ \phi_1]$.  For each $\sigma \in S_n$ we choose a representative of $\sigma [\phi_n\circ... \circ \phi_1]$, as a sequence of $n$ degree $1$ morphisms, which we denote as 
$\phi_{\sigma , \sigma^{-1}(n)}\circ ... \circ \phi_{\sigma , \sigma^{-1}(1)}$.  This notation is meant to reflect the case where $\phi_i$ contracts an edge of a graph, which by convention was the $i^{th}$ edge.  Then $\phi_{\sigma, i}$ contracts the same edge as $\phi_i$, but it does so in the order dictated by $\sigma$.  

Fix a nonempty subset $X\subset \{1\cdc n\}$ of size $p$ and let $Y:=\{1\cdc n\} \setminus X$.  Define
$S_{Y<X}:= \{\sigma \in S_n \ | \ \sigma(y) <\sigma (x) \ \forall \ y\in Y, x\in X \}.$
 An element in $S_{Y<X}$ specifies and is specified by a total order of the sets $X$ and $Y$.  Fix $\sigma \in S_{Y<X}$ and consider the associated chain of degree $1$ morphisms
 \begin{equation}\label{chain1}
 (\phi_{\sigma , x_p}\circ \dots \circ \phi_{\sigma , x_1})\circ (\phi_{\sigma , y_{n-p}}\circ ... \circ \phi_{\sigma , y_1}).
 \end{equation}  
  Define $\phi_{\sigma,X}$ and $\phi_{\sigma, Y}$ to be the composition of the morphisms indexed by $X$ and $Y$ respectively in Equation $\ref{chain1}$.  
 The morphism $\phi_{\sigma,X}$ is a basic morphism of degree $p$.    The morphism $\phi_{\sigma, Y}$ may not be a basic morphism.  Suppose it has outputs $v_1\tensor\dots\tensor v_b$, 
 then by the Feynman category axioms, $\phi_{\sigma, Y} =  \tensor_{i=1}^b \phi_{\sigma, Y}^i$ where each $\phi_{\sigma, Y}^i$ is a basic morphism of degree $\leq n-p < n$.

Any decomposition of $\phi$, for example the chain in Equation $\ref{chain1}$, determines a $\mathsf{V}$-colored tree by reading the composition as a flow chart.  The vertices in such a tree are in bijective correspondence with the morphisms in the chain and hence are totally ordered.  Passing to the quotient by even permutations of such an order, we may say such a chain determines an homogeneous element in $C_\ast(\phi)$.  Moreover this element is independent of the choice of representative of the chain, since the composition of degree $0$ morphisms in such a chain produces the same element in the coinvariants (see Equation $\ref{coinvariants}$).

With this in mind we define the chain complex $C_\ast(\phi, X) \subset C_{p+\ast}(\phi)^p/C_{p+\ast}(\phi)^{p-1}$
to be the span of all such homogeneous elements associated to compositions of morphisms inside the parentheses in the chain
\begin{equation}\label{chain}
\phi_{\sigma,X}\circ (\phi_{\sigma , y_r}\circ ... \circ \phi_{\sigma , y_1}),
\end{equation}
over all $\sigma \in S_{Y<X}$.  In particular, if $Y\neq \emptyset$ then $C_q(\phi,X)$ is concentrated between degrees $-2-p\geq q \geq -n-1$. This upper bound need not be achieved (unlike in Lemma $\ref{Plem}$), as the maximum non-zero dimension is spanned by $\phi_{\sigma , X}\circ \phi_{\sigma , Y}$, which has degree $-b-1-p$.  The complex $C_\ast(\phi,X)$ is to be understood as the span of trees whose root is labeled by a composition of the morphisms indexed by $X$. 

The fact that each $E^0_{p,\ast}$ splits follows from the fact that the $E^0$ differential does not alter the label of the root vertex.   
 To complete the proof of statement (1), define
\begin{equation}\label{iso}
\tensor_{i} C_\ast(\phi_{\sigma, Y}^i)\to \Sigma^{p+1} C_\ast(\phi,X) 
\end{equation}
to take a pure tensor of homogeneous vectors and add the root vertex labeled by $\phi_{X}$. Regarding the degrees, note that starting with a pure tensor in the source of cellular degree $\sum r_i = r$, the image satisfies $p+q=-(r+1)$ and hence the map in Equation $\ref{iso}$ has degree 0.

It is straight-forward to see that the map in $\ref{iso}$ is an isomorphism.  The terms in the differential are exactly the same (since the $E^0$ differential can't expand the root vertex) and the map is a bijection on the standard basis.  Suppose that each $\phi_{\sigma, Y}^i$ is a morphism of degree $n_i$.  By the induction hypothesis applied to statement (4) of the Lemma we conclude
\begin{equation*}
H_\ast(C_\ast(\phi,X)) \cong  \Sigma^{-p-1}(\tensor_i (\Sigma^{-n_i} k )) \cong \Sigma^{-p-1}(\Sigma^{p-n} k) \cong \Sigma^{-n-1} k
\end{equation*}
as graded vector spaces.  This establishes claim (1).  Claim (2) then follows immediately.  In particular, every homology class in $E^1_{p,-n-1}$ is represented up to scalar by the choice of a subset $X\subset \oneton$.

To complete the proof, observe that the terms in the differential $d^1\colon E^1_{p,-n-1}\to E^1_{p-1,-n-1}$ are indexed by ways to expel an element from the set $X$. By convention for the differential in the cobar construction, the signs are determined by adding the expelled element in the right-most position of the mod 2 ordered set $Y=X^c$.  These signs coincide with the signs in the differential of the odd simplicial chain complex of a simplex, and so invoking Equation $\ref{deltaiso}$ yields statement (3).  Finally we observe statement (4) follows immediately from statement (3).
\end{proof} 

\subsection{Conclusion of the proof.}
It remains to observe that Theorem $\ref{main}$ follows immediately from Lemma $\ref{widthfiltration}$.  From the lemma we know
	\begin{equation*}
H_r(C_\ast(\phi)) = \begin{cases} k & \text{ if } r=-n \\ 0 & \text{ else } \end{cases}.
\end{equation*}
Taking the homology of the morphism $\Omega(\mathbb{F}^\ast)(\vec{v})\to \mathbb{F}^!(\vec{v})$, yields the morphism of graded vector spaces:
\begin{equation*}
H_\ast(\Omega(\mathbb{F}^\ast)(\vec{v})) \cong \bigoplus_{\phi \in \mathsf{F}(\vec{v})} H_\ast(C_\ast(\phi)) \cong \bigoplus_{\phi \in \mathsf{F}(\vec{v})}\Sigma^{-|\phi|}k\to \mathbb{F}^!(\vec{v})
\end{equation*}
and since homology classes on the right hand side are represented by mod 2 orders of the set of degree $1$ morphisms, this map is not zero on each summand, and hence is an isomorphism as desired.

\section{From Permutohedra to Simplices Revisited.}

Let us revisit the surjective quasi-isomorphism $C_\ast(P_n)\twoheadrightarrow C_\ast(\Delta^
{n-1})$ of Corollary $\ref{PtoS}$.  This map may be viewed as an algebraic analog of blowing down a permutohedron to form a simplex, or dually truncating a simplex to form a permutohedron.  The following Lemma shows that the fiber of the Koszul map $C_\ast(\phi)$ over a degree $n$ morphism $\phi$ fits in-between, and may be combinatorially realized both as a blow-up of a simplex and as a blow-down of a permutohedron.  In particular, the permutohedron may be considered the maximal possible fiber.  A particular example in the case of cyclic (or modular) operads is illustrated in Figure $\ref{fiber}$ and $\ref{fiber2}$, see \cite{WardMP} for more details in that case.  In general we have the following result:

\begin{figure}
	\centering
	\includegraphics[scale=.75]{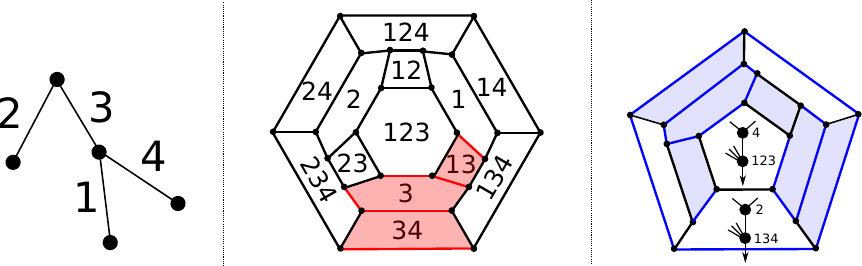}
	\caption{
		Left: A degree 4 morphism $\phi$ in the Feynman category encoding cyclic operads, plus an order on the set of edges. Center: the cell complex of the 3d permutohedron $P_4$ (in top down view).  Only the least block of the codimension 1 cells are labeled, with the label of the ``bottom'' hexagon $\{1,2,3\} >\{4\}$ suppressed.  Right: the complex $C_\ast(\phi)$ with most labels suppressed.  The kernels of the maps in Lemma $\ref{maps}$ are colored red and blue respectively.}	%
	\label{fiber}
\end{figure}

\begin{lemma}\label{maps}  For every $\phi$ there are surjective quasi-isomorphisms of filtered complexes $
	C_\ast(P_n) \stackrel{\sim}\twoheadrightarrow 
	\Sigma^{n}C_\ast(\phi) \stackrel{\sim}\twoheadrightarrow C_\ast(\Delta^
	{n-1}) $
	factoring the morphism in Corollary $\ref{PtoS}$.
\end{lemma}
\begin{proof}  
	As above, $\phi$ is a degree $n$ basic morphism in a cubical Feynman category and we fix a decomposition $\phi = \phi_n\circ\dots\circ\phi_1$ into degree $1$ morphisms.	Here $C_\ast(P_n)$ and
	$\Sigma^{n}C_\ast(\phi)$ are both filtered by width, as above, and $C_\ast(\Delta^{n-1})$ is filtered by the number of elements in the set corresponding to a cell (which is just the cellular degree + 1).

The map $\Sigma^{n}C_\ast(\phi) \to C_\ast(\Delta^{n-1})$ is defined as follows. From the proof of Lemma $\ref{widthfiltration}$, the complex $C_\ast(\phi)$ is spanned by trees labeled by basic morphisms, and each such tree belongs to some complex $C_\ast(\phi,X)$ for some subset $X\subset \oneton$.  If such a tree has $n-p$ edges we first map it to the chain in $C^-_\ast(\Delta^{n-1})$ corresponding to the complement of $X$, else we define its image to be $0$.  The fact that this is a surjective quasi-isomorphism then follows immediately from Lemma $\ref{widthfiltration}$.  We then compose with the isomorphism of Equation $\ref{deltaiso}$ as in Corollary $\ref{PtoS}$.

The filtered quasi-isomorphism $C_\ast(P_n)\to \Sigma^{n}C_\ast(\phi)$ is defined as follows. Let $\alpha$ be a cell of degree $r$, i.e.\ a TO partition of $\oneton$ with $n-r$ blocks. Choose a compatible total order of $\oneton$ by ordering the elements of each block.  To this total order we associate the permutation $\sigma\in S_n$ and hence a chain of morphisms $\sigma(\phi_n\circ\dots\circ\phi_1)$.  Reading this chain of morphisms as a flow chart gives us an element in $C_{-n}(\phi)$ for which each vertex is labeled $\phi_{\sigma, i}$.  Contract all tree edges which join vertices labeled by $\phi_{\sigma, i}$ and $\phi_{\sigma, j}$ for which $i$ and $j$ were in the same block of $\alpha$.  Call the resulting element $\mathsf{t}_\alpha \in C_\ast(\phi)$.  Observe that $\mathsf{t}_\alpha$ does not depend on the choice of total order of each block of $\alpha$. 

We then define $C_\ast(P_n)\to \Sigma^{n}C_\ast(\phi)$ by sending $\alpha$ to $\mathsf{t}_\alpha$ if $\mathsf{t}_\alpha \in C_{r-n}(\phi)$ and to $0$ otherwise.   Intuitively, each vertex of $\mathsf{t}_\alpha$ is labeled by a composition of morphisms $\phi_{\sigma, \sigma^{-1}(i)}$ and so specifies a partition of $\oneton$.  The map is non-zero precisely when this partition agrees with $\alpha$ and is zero if it is finer than $\alpha$.

This map is surjective and preserves width.  Let's show this map is dg.  First, suppose $\alpha$ maps to $0$.  Then the partition of $\oneton$ given by the vertex labels of $\mathsf{t}_\alpha$ is finer than $\alpha$.  In this case all differential terms of $\alpha$ also map to $0$, with one possible exception.  Namely, if $\mathsf{t}_\alpha$ has exactly one additional block, then there are two terms in $d(\alpha)$ whose image is non-zero.  They are given by splitting the block of $\alpha$ according to $\mathsf{t}_\alpha$ in both possible orders.  However, since these map to the same tree but with a transposed edge order, their sum maps to zero as desired.

Now consider an $\alpha$ which does not map to $0$. 
On the one hand, every term of $d(\mathsf{t}_\alpha)$ corresponds to a unique term in $d(\alpha)$.  Conversely terms in $d(\alpha)$ which do not appear in $d(\mathsf{t}_\alpha)$ must map to zero.
Indeed, any such term in $d(\alpha)$ sub-divides a block and if such a term maps to something non-zero then,  considering a total order on this block compatible with this sub-division, it must be the case that the two sub-blocks are joined by an edge in the associated tree, and hence correspond to a term in $d(\mathsf{t}_\alpha)$.  Hence this map is dg.

The composition of the two maps constructed above gives a map $C_\ast(P_n) \to C_\ast(\Delta_{n-1})$ which sends a TO partition to the cell indexed by the complement of its least block, and hence agrees with the map given in Corollary $\ref{PtoS}$.  Finally, we observe that the these two maps and their composition are all quasi-isomorphisms by the $2$ out of $3$ property.  \end{proof}

When $\phi$ is a graph of vertices and edges arranged in a line, the fiber $C_\ast(\phi)$ is an associahedron and this map from $P_n$ was constructed by Tonks in \cite{Tonks}.

\begin{figure}
	\centering
	\includegraphics[scale=.75]{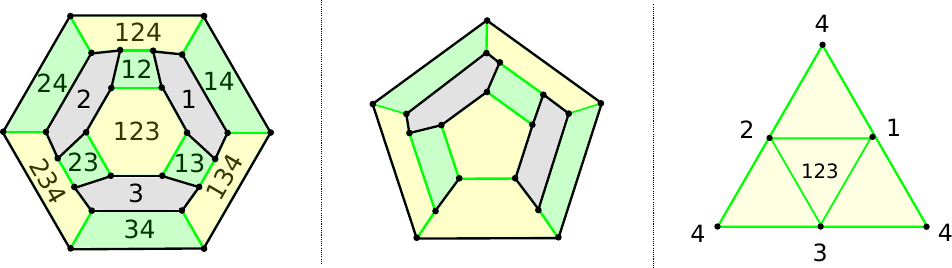}
	\caption{  The maps in Lemma $\ref{maps}$ respect the width filtration.  Left and center are as in Figure $\ref{fiber}$; right is the simplex $\Delta^3$.  Width 1 is in black, width 2 is in green, width 3 in yellow. }
	\label{fiber2}
\end{figure}

\section{Outlook.} 
Having established that the linearization $\mathbb{F}$ of a cubical Feynman category $\mathsf{F}$ is Koszul, we define $\mathbb{F}_\infty:= \Omega((\mathbb{F}^!)^\ast)$ and refer to an $\mathbb{F}_\infty$-op as a strongly homotopy $\mathsf{F}$-op or a weak $\mathsf{F}$-op.  Since $\mathbb{F}$ is Kozsul, the category of weak $\mathsf{F}$-ops admits a homotopy transfer theory and a bar-cobar duality in analogy with the classical setting.  With this in mind we offer, by way of conclusion, several potential applications and directions for future study.


{\bf I.} Consider $\mathsf{F}$ to be the Koszul Feynman category encoding the colored operad which in turn encodes the swiss cheese operad.  Since the swiss cheese operad is known not to be formal \cite{Livernet}, its homology must carry higher operations which assemble to a weak $\mathsf{F}$-op and encode its rational homotopy type.  It would be interesting to determine what concretely can be said about them.

{\bf II.} Define a weak $\mathfrak{K}$-twisted $\mathsf{F}$-op as a representation of $(\mathbb{F}^!)_\infty:= \Omega(\mathbb{F}^\ast)$.  Such a representation is a functor whose limit should carry an odd $L_\infty$ algebra, generalizing the odd Lie structure in the strict case.  One can ask if the Maurer-Cartan functor is still representable in the weak case via the bar construction of the monoidal unit.  In the presence of a non-connected multiplication, this $L_\infty$-algebra may be upgraded to a class of $\mathsf{BV}_\infty$ algebras, generalizing \cite{KWZ}.

	
{\bf III.} There are cubical Feynman categories whose morphisms are not graphs in any naive sense such the Feynman category encoding associative algebras, for which our Koszulity result is equivalent to the contractibility of associahedra \cite{Stasheff}. It would be interesting to investigate the cubicality condition in related examples, such as shuffle algebras, permutads and twisted associative algebras.  The presentation of permutads given in \cite[Definition 3]{Permutads} could be used to establish cubicality, which would provide an alternate approach to \cite[Theorem B]{Permutads}.

{\bf IV.} More generally, Batanin and Markl have identified in \cite{MB3} a large class of their operadic categories which are Koszul, see \cite[Theorem 11.5]{MB3}.  A precise comparison of their criteria with our cubical hypothesis, as well as a translation of the notion of cubicality into the language of operadic categories would be very desirable.

{\bf V.} The connection to the permutahedral decomposition of Cacti \cite{KZhang} and the $B_+$ and $B_-$ operators is intriguing. This should provide connections to Steenrod operations via the formulation \cite{KMM}.
	
{\bf VI.} The connection to the plus constructions \cite{feynmanrep} and W--constructions, \cite{BM,KW} will be studied in forthcoming work.


\end{document}